\documentclass[12pt]{amsart}
\usepackage[english]{babel}
\usepackage[T1]{fontenc}
\usepackage[latin1]{inputenc}
\usepackage{lmodern}
\usepackage{a4}
\usepackage[all]{xy}
\usepackage{verbatim}
\usepackage{stmaryrd}
\usepackage{amsfonts}
\usepackage{amsmath,amssymb,amsthm}
\usepackage{enumerate}
\usepackage{xspace}
\usepackage{ulem}
\usepackage{euscript}
\usepackage{epsfig}
\usepackage{pxfonts}
\usepackage{mathrsfs}
\usepackage{indentfirst}
\usepackage[colorlinks,linkcolor=red,citecolor=blue]{hyperref}

\makeatletter\@addtoreset{equation}{section}

\newtheorem{theorem}{Theorem}[section]

\newtheorem{lemma}[theorem]{Lemma}

\newtheorem{proposition}[theorem]{Proposition}
\theoremstyle{remark}

\newtheorem{remark}[theorem]{Remark}

\newcommand{\BH}{{\mathcal B}(\H)}

\newcommand{\FH}{{\mathcal F}(\mathcal{H})}

\newcommand{\NHO}{{\mathcal N}_1(\mathcal{H})}
\newcommand{\UH}{{\mathcal U}(\mathcal{H})}
\newcommand{\C}{\mathbb{C}}

\renewcommand{\H}{\mathcal{H}}

\newcommand{\spn}{{\rm span}}

\newcommand{\ran}{{\rm ran}\,  }
\renewcommand{\ker}{{\rm ker}\,}
\newcommand{\simu}{\sim_u}
\newcommand{\rank}{\operatorname{rank}\,}
\renewcommand{\r}{{r}\,  }
\newcommand{\w}{{w}\,  }

\newcommand{\tr}{\mathrm{t}}
\newcommand{\tzh}{{\mathcal F}_0(\H)}

\title[]{Spectral and numerical radius preservers on pairs of unitarily similar operators}

\author[Z. Amara]{Zouheir Amara}
\address{Department of Mathematics\\Laboratoire Math\'ematiques et Applications, FSO\\Mohammed First University\\Oujda 60000\\Morocco}
\email{z.amara@ump.ac.ma}


\begin{document}

\begin{abstract}
Let $\H$ be a complex Hilbert space. We characterize the bijective linear maps $\Phi:\BH\to\BH$ that preserve the spectral radius, respectively the numerical radius, on pairs of unitarily similar operators. We also characterize bijective linear maps that transform unitary similarity into similarity. Our results cover both finite and infinite-dimensional Hilbert spaces, with no separability assumption in the latter case.
\end{abstract}

\maketitle

\section{Introduction}

Throughout, let $\H$ be a complex Hilbert space, with $\dim\H\geq2$, and let $\BH$ be the algebra of all bounded linear operators on $\H$. Recall that two operators $T,S\in\BH$ are said to be {\it similar} if there exists an invertible operator $A\in\BH$ such that $T=ASA^{-1}$, and they are said to be {\it unitarily similar} if there exists a unitary operator $W\in\BH$ such that $T=WSW^*$. We write $T\sim S$ and $T\simu S$, respectively. The spectrum and spectral radius of $T$ are denoted by $\sigma(T)$ and $r(T)$, respectively. Its numerical radius is
$$
w(T)=\sup\big\{|\langle Tx,x\rangle|:\|x\|=1\big\}.
$$

\medskip

Several recent linear preserver problems concern functions that are constant on prescribed equivalence classes of matrices or operators. More precisely, let $\mathcal R$ be an equivalence relation on $\BH$ and let $\eta$ be a function that is constant on the $\mathcal R$-classes. One may then ask for the structure of linear maps $\Phi$ on $\BH$ satisfying
\begin{equation}\label{R.to.eta}
T\,\mathcal R\,S
\implies
\eta(\Phi(T))=\eta(\Phi(S))
\qquad\text{for all $T,S\in\BH$}.
\end{equation}

Costara \cite{Costara} considered this problem for the spectral radius on pairs of similar operators. In \cite{AmaraOudghiri}, we studied the corresponding problem on pairs of unitarily similar operators for the reduced minimum modulus, the minimum modulus, the surjectivity modulus and the norm. Bourhim and Mabrouk \cite{BourhimMabrouk} recently treated the numerical radius on pairs of unitarily similar operators.

\medskip

There is a natural connection between such function-preserver problems and preservers of equivalence relations. Suppose that $\eta$ is invariant under an equivalence relation $\mathcal S$. Then every map sending $\mathcal R$-equivalent operators to $\mathcal S$-equivalent operators necessarily satisfies \eqref{R.to.eta}. It is then natural to ask under what assumptions \eqref{R.to.eta} forces $\Phi$ to preserve or transform an underlying equivalence relation. This connection is used in \cite{AmaraOudghiri,BourhimMabrouk,Costara} and will also play a central role here.

\medskip

The corresponding equivalence-relation preservers have themselves been extensively studied. In finite dimension, Hiai \cite{Hiai} and Lim \cite{Lim} described linear maps on matrix algebras preserving similarity. For an infinite-dimensional separable complex Hilbert space, the corresponding problem for bijective linear maps on $\BH$ was solved by \v{S}emrl \cite{Semrl}. Lu and Peng \cite{LuPeng} subsequently extended this result to bijective linear maps on $\mathcal{B}(X)$, where $X$ is an arbitrary infinite-dimensional complex Banach space. For unitary similarity, Petek \cite{Petek} characterized surjective linear maps preserving unitary similarity in both directions in the setting of an arbitrary infinite-dimensional complex Hilbert space, while Karder et al. \cite{KarderPetekTaghavi} obtained the corresponding one-direction result for bijective linear maps on a separable infinite-dimensional complex Hilbert space. Their proof relies on the theorem of Davidson and Marcoux concerning linear spans of unitary and similarity orbits \cite{DavidsonMarcoux}, and consequently requires separability.

\medskip

In this paper, we first consider the spectral radius on unitary similarity classes. More precisely, replacing similarity in Costara's problem \cite{Costara} with the finer equivalence relation of unitary similarity, we study bijective linear maps $\Phi:\BH\to\BH$ satisfying
$$
T\simu S
\implies
r(\Phi(T))=r(\Phi(S))
\qquad\text{for all $T,S\in\BH$}.
$$
We prove that, for bijective $\Phi$, this condition is equivalent to requiring that $\Phi$ transform unitary similarity into similarity.

\medskip

Our second objective is the corresponding problem for the numerical radius:
$$
T\simu S
\implies
w(\Phi(T))=w(\Phi(S))
\qquad\text{for all $T,S\in\BH$}.
$$
In this case, the function-preserver problem is closely connected with the preservation of unitary similarity itself. This allows us to obtain the complete classification in both finite and infinite dimensions. In particular, our infinite-dimensional results remove the separability assumptions from the results of Bourhim and Mabrouk \cite{BourhimMabrouk} and Karder et al. \cite{KarderPetekTaghavi}.

\medskip

The remainder of the paper is organized as follows. Section 2 states the main results. Section 3 contains the preliminary results used throughout the paper. In Section 4, we prove the spectral radius results, and Section 5 is devoted to the numerical radius case.

\section{Main results}

For $T\in\BH$, let $T^{\tr}$ denote its transpose with respect to a fixed orthonormal basis of $\H$, and denote the identity operator on $\H$ by $I_{\H}$. When $\H$ is finite-dimensional, ${\rm Tr}(T)$ denotes the usual trace.

\begin{theorem}\label{main.thm.1}
Let $\H$ be a finite-dimensional complex Hilbert space, and let $\Phi:\BH\to\BH$ be a bijective linear map. The following are equivalent:
\begin{enumerate}[\rm (i)]
\item $\Phi(T)\sim\Phi(S)$ whenever $T\simu S$.
\item $\r(\Phi(T))=\r(\Phi(S))$ whenever $T\simu S$.
\item There exist an invertible operator $A\in\BH$ and scalars $c,\alpha\in\C$, with $c\neq0$, such that either
$$
\Phi(X)=cAXA^{-1}+\alpha {\rm Tr}(X)I_{\H}\qquad\text{for every }X\in\BH,
$$
or
$$
\Phi(X)=cAX^{\tr}A^{-1}+\alpha {\rm Tr}(X)I_{\H}\qquad\text{for every }X\in\BH.
$$
\end{enumerate}
\end{theorem}

\begin{theorem}\label{main.thm.2}
Let $\H$ be an infinite-dimensional complex Hilbert space, and let $\Phi:\BH\to\BH$ be a bijective linear map. The following are equivalent:
\begin{enumerate}[\rm (i)]
\item $\Phi(T)\sim\Phi(S)$ whenever $T\simu S$.
\item $\r(\Phi(T))=\r(\Phi(S))$ whenever $T\simu S$.
\item There exist an invertible operator $A\in\BH$ and a nonzero scalar $c\in\C$ such that either
\begin{equation}\label{standard.form1.spectral}
\Phi(X)=cAXA^{-1}\qquad\text{for every }X\in\BH
\end{equation}
or
\begin{equation}\label{standard.form2.spectral}
\Phi(X)=cAX^{\tr}A^{-1}\qquad\text{for every }X\in\BH.
\end{equation}
\end{enumerate}
\end{theorem}

We next turn to the numerical radius.

\begin{theorem}\label{main.thm.3}
Let $\H$ be a finite-dimensional complex Hilbert space, and let $\Phi:\BH\to\BH$ be a bijective linear map. The following are equivalent:
\begin{enumerate}[\rm (i)]
\item $\Phi(T)\simu\Phi(S)$ whenever $T\simu S$.
\item $\w(\Phi(T))=\w(\Phi(S))$ whenever $T\simu S$.
\item There exist a unitary operator $W\in\BH$ and scalars $c,\alpha\in\C$, with $c\neq0$, such that either
$$
\Phi(X)=cWXW^*+\alpha {\rm Tr}(X)I_{\H}\qquad\text{for every }X\in\BH,
$$
or
$$
\Phi(X)=cWX^{\tr}W^*+\alpha {\rm Tr}(X)I_{\H}\qquad\text{for every }X\in\BH.
$$
\end{enumerate}
\end{theorem}

\begin{theorem}\label{main.thm.4}
Let $\H$ be an infinite-dimensional complex Hilbert space, and let $\Phi:\BH\to\BH$ be a bijective linear map. The following are equivalent:
\begin{enumerate}[\rm (i)]
\item $\Phi(T)\simu\Phi(S)$ whenever $T\simu S$.
\item $\w(\Phi(T))=\w(\Phi(S))$ whenever $T\simu S$.
\item There exist a unitary operator $W\in\BH$ and a nonzero scalar $c\in\C$ such that either
\begin{equation}\label{standard.form1.numerical}
\Phi(X)=cWXW^*\qquad\text{for every }X\in\BH
\end{equation}
or
\begin{equation}\label{standard.form2.numerical}
\Phi(X)=cWX^{\tr}W^*\qquad\text{for every }X\in\BH.
\end{equation}
\end{enumerate}
\end{theorem}

In each of Theorems \ref{main.thm.1}--\ref{main.thm.4}, the implications {\rm (i)}$\implies${\rm (ii)} and {\rm (iii)}$\implies${\rm (i)} are immediate. Hence it remains to prove {\rm (ii)}$\implies${\rm (iii)}.

\section{Preliminary results}

Denote by $\UH$ the set of all unitary operators on $\H$.

\begin{proposition}\label{T.scalar}
Let $T\in\BH$ be such that $UTU^*\in\C T$ for every $U\in\UH$. Then $T\in\C I_{\H}$.
\end{proposition}

\begin{proof}
Assume without loss of generality that $T\neq 0$. Then, for every $U\in\UH$, there exists $\alpha_U\in\C$, with $|\alpha_U|=1$, such that $UTU^*=\alpha_U T$. This implies that, for every $U\in\UH$,
$$
UT^*TU^*=T^*T\quad\text{and}\quad UTT^*U^*=TT^*,
$$
equivalently, $UT^*T=T^*TU$ and $UTT^*=TT^*U$. Thus, since $\BH$ is spanned by $\UH$, $T^*T, TT^*\in\C I_{\H}\setminus\{0\}$; in particular, $T$ is invertible. Then
\begin{equation}\label{VUV*}
T^{-1}UT=\alpha_U U\quad\text{for every }U\in\UH.  
\end{equation}
Now, fix a unimodular $\lambda\in\C$ so that $\lambda^2\neq 1$, let $P\in\BH$ be a rank-one orthogonal projection, and set $U=I_{\H}+(\lambda-1)P$. Then, $U$ is unitary and \eqref{VUV*} implies that
$$
\{1,\lambda\}=\sigma(U)=\alpha_U\sigma(U)=\{\alpha_U,\alpha_U \lambda\}.
$$
Then $\alpha_U=1$. In fact, if $\alpha_U=\lambda$ and $\alpha_U \lambda=1$, we would have $\lambda^2=1$; a contradiction. It follows that $T^{-1}UT=U$, and hence $UT=TU$. This implies that $PT=TP$, and so since $P$ was arbitrary, we obtain $T\in\C I_{\H}$.
\end{proof}

Let $\FH$ denote the linear space of all finite-rank operators on $\H$. For $T\in\BH$, $\ran T$ denotes the range of $T$.

\begin{proposition}\label{rank.proposition}
Assume that $\H$ is infinite-dimensional, and let $T\in\BH$. Suppose that for every rank-two orthogonal projection $P\in\BH$, we have
\begin{equation}\label{rank.bound.2}
\rank PT(I_{\H}-P) + \rank (I_{\H}-P)TP \leq 2.    
\end{equation}
Then, there exist $\lambda\in\C$ and an operator $R\in\BH$ of rank at most one such that $T=\lambda I_{\H}+R$.
\end{proposition}

\begin{proof}
Without loss of generality, we may assume that $T\notin \C I_{\H}$. The proof is divided into two steps:

\medskip

{\bf Step 1.} We prove that $T=\lambda I_{\H} + F$ for some $\lambda\in\C$ and $F\in \FH$. According to \cite[Lemma 2.2]{LuPeng}, it suffices to prove that, for all $x,y\in \H$, the vectors
$$
x,\ y,\ Tx,\ Ty
$$
are linearly dependent. Assume, to the contrary, that there exist vectors $x,y\in\H$ such that
$x, y, Tx, Ty$
are linearly independent. Fix an arbitrary nonzero vector $z\in \{x,y\}^\perp$, and let $w\in \{x,y\}^\perp$ and $a,b\in\C$ be such that
\begin{equation}\label{T^*z}
T^*z=ax+by+w.  
\end{equation}
Elementary linear algebra shows that there exists $r>0$ such that, for all $t\in[0,r)$, the vectors
$$
x,\ y+tz,\ Tx,\ T(y+tz)
$$
remain linearly independent. Set
$$
M_t=\spn\{x,y+tz\}
\quad\text{and}\quad
Q_t=I_{\H}-P_t
$$
where $P_t\in\BH$ is the orthogonal projection onto $M_t$.
Then, it is easy to see that, for every $t\in[0,r)$, the vectors $Q_tTx$ and $Q_tT(y+tz)$ are linearly independent, and hence
$
\rank Q_tTP_t=2.
$
By hypothesis, we get
$P_tTQ_t=0$; equivalently $Q_tT^*P_t=0$. In other words,
\begin{equation}\label{M_t.invariant}
T^*M_t\subseteq M_t\qquad\text{for every }t\in[0,r).
\end{equation}
Since $\bigcap_{0\leq t<r}M_t=\C x$, there exists $\alpha\in\C$ such that
$$
T^*x=\alpha x.
$$
Applying the same argument to the subspace $N_t:=\spn\{y,x+tz\}$ gives $s>0$ such that
\begin{equation}\label{N_t.invariant}
 T^*N_t\subseteq N_t\qquad\text{for every }t\in[0,s),   
\end{equation}
and so there exists $\beta\in\C$ such that
$$
T^*y=\beta y.
$$

It follows from \eqref{T^*z} that
$$
T^*(y+tz)
=
\beta y+tT^*z
=
tax+(\beta+tb)y+t w.
$$
Moreover, by \eqref{M_t.invariant}, for every $t\in[0,r)$, there exist $c_t,d_t\in\C$ such that $T^*(y+tz)=c_t x+d_t(y+tz)$, and so
$$
tax+(\beta+tb)y+t w
=
c_t x+d_t(y+tz).
$$
This implies that
$$
d_t=\beta+tb
\qquad\text{and}\qquad
tw=d_t tz,
$$
and hence
$$
w=d_t z=(\beta+tb)z\qquad\text{for every }t\in (0,r).
$$
Noting that $w$ does not depend on $t$, we get $b=0$, and therefore $w=\beta z$.

Since $T^*x=\alpha x$ and, by \eqref{N_t.invariant}, $T^*(x+tz)\in N_t$ for every $t\in [0,s)$, the same argument as above leads to  $a=0$ and $w=\alpha z$. It follows that $\alpha=\beta$ and $w=\alpha z$. Thus, by \eqref{T^*z}, 
$$
T^*z=\alpha z.
$$
Therefore, since $z\in\{x,y\}^{\perp}$ was arbitrary, we obtain $T^*=\alpha I_{\H}$, contradicting the fact that $T\notin\C I_{\H}$. Consequently, $x,y,Tx,Ty$ are linearly dependent for all $x,y\in\H$.

\medskip

{\bf Step 2.} We prove that $\rank F\leq 1$.
Suppose for a contradiction that $\rank F\ge2$. Choose $u,v\in\H$ so that $Fu,Fv$ are linearly independent. Set $L=\ran F$.
Since $F$ has finite rank, $\ker F$ is infinite-dimensional. Hence, elementary linear algebra shows that we can choose $x,y\in\ker F$ such that $u+x+L$ and $v+y+L$ are linearly independent.

Set $a=u+x$ and $b=v+y$. Then
$
Fa=Fu,
$
$
Fb=Fv,
$
and $Fa,Fb$ are linearly independent and belong to $L$. It follows that the vectors $a,\ b, \ Fa, \ Fb$
are linearly independent. Since $T=\lambda I_{\H}+F$, we have
$$
\spn\{a,b,Ta,Tb\}
=
\spn\{a,b,Fa,Fb\},
$$
and so we conclude that $a,\ b,\ Ta,\ Tb$ are linearly independent; a contradiction. Consequently, $F$ is of rank at most $1$.
\end{proof}

\begin{remark}
If $P\in\BH$ is an orthogonal projection, then
$$
PT-TP
=
\begin{bmatrix}
0&PT_{|\ran(I_{\H}-P)}\\
-(I_{\H}-P)T_{|\ran P}&0
\end{bmatrix}
\begin{array}{l}
\ran P\\
\ran(I_{\H}-P)
\end{array}.
$$
This implies that $\rank (PT-TP)=\rank PT(I_{\H}-P) + \rank (I_{\H}-P)TP$. Thus, the condition in \eqref{rank.bound.2} may equivalently be replaced by $\rank (PT-TP)\leq2$.
\end{remark}

\begin{remark}
In the proof of our main results, Proposition \ref{rank.proposition} is used
only in the following weaker form: if \eqref{rank.bound.2} holds for every
finite-rank orthogonal projection $P$, then $T$ is a scalar operator plus a
finite-rank operator. We stated Proposition \ref{rank.proposition} in the
sharper form above because both the conclusion that the finite-rank part has
rank at most one and the fact that it suffices to consider rank-two orthogonal projections
seem to be of independent interest.
\end{remark}

\section{Proof of Theorems \ref{main.thm.1} and \ref{main.thm.2}}

In this section, we assume that $\Phi:\BH\to\BH$ is a bijective linear map that satisfies
$$
T\simu S \implies \r(\Phi(T))=\r(\Phi(S))\qquad\text{for all $T,S\in\BH$}.
$$

\medskip

The proof of the following lemma is inspired by \cite[Proposition 3.1]{AmaraOudghiri}.

\begin{lemma}\label{A_U.spectral}
For every $U\in\UH$, there
exists an invertible operator $A_U\in\BH$ such that
$$
\Phi(UXU^*)=A_U\Phi(X)A_U^{-1}
\qquad\text{for every }X\in\BH.
$$
\end{lemma}

\begin{proof}
Fix $U\in\UH$, and choose $V\in\UH$ such that
$V^2=U$. Define $\varphi_V:\BH\to\BH$ by
$$
\varphi_V(X)=\Phi\bigl(V\Phi^{-1}(X)V^*\bigr).
$$
Clearly, $\varphi_V$ is a bijective linear map. Since
$
V\Phi^{-1}(X)V^*\simu \Phi^{-1}(X),
$
we have
$$
r(\varphi_V(X))=r\left( \Phi\bigl(V\Phi^{-1}(X)V^*\bigr) \right)=r\left( \Phi\bigl(\Phi^{-1}(X)\bigr) \right)=r(X)
$$
for every $X\in\BH$. By \cite{BresarSemrl}, there exist an invertible operator $S_V\in\BH$
and a unimodular scalar $\alpha_V\in\C$ such that either
\begin{equation}\label{varphi.form1}
\Phi(VXV^*)=\varphi_V(\Phi(X))=\alpha_VS_V\Phi(X)S_V^{-1}
\qquad\text{for every }X\in\BH,
\end{equation}
or
\begin{equation}\label{varphi.form2}
\Phi(VXV^*)=\varphi_V(\Phi(X))=\alpha_VS_V\Phi(X)^{\tr}S_V^{-1}
\qquad\text{for every }X\in\BH.
\end{equation}

If \eqref{varphi.form1} holds, then, for every $X\in\BH$,
$$
\begin{aligned}
\Phi(UXU^*)
&=\Phi\bigl(V(VXV^*)V^*\bigr)\\
&=\alpha_VS_V\Phi(VXV^*)S_V^{-1}\\
&=\alpha_V^2S_V^2\Phi(X)(S_V^2)^{-1}.
\end{aligned}
$$
If \eqref{varphi.form2} holds, then, for every $X\in\BH$,
$$
\begin{aligned}
\Phi(UXU^*)
&=\alpha_VS_V\Phi(VXV^*)^{\tr}S_V^{-1}\\
&=\alpha_V^2S_V(S_V^{-1})^{\tr}\Phi(X)S_V^{\tr}S_V^{-1}.
\end{aligned}
$$
Consequently, in either case there exist an invertible operator
$A_U\in\BH$ and a unimodular scalar $\beta_U\in\C$ such that
\begin{equation}\label{beta_U.spectral}
\Phi(UXU^*)=\beta_UA_U\Phi(X)A_U^{-1}
\qquad\text{for every }X\in\BH.
\end{equation}

It remains to prove that $\beta_U=1$ for every $U\in\UH$. Set $T=\Phi^{-1}(I_{\H})$. From \eqref{beta_U.spectral}, we get
$$
\Phi(UTU^*)=\beta_U A_UA_U^{-1}=\beta_U I_{\H}=\Phi(\beta_UT),
$$
and so
$$
UTU^*=\beta_UT
\qquad\text{for every }U\in\UH.
$$
Proposition \ref{T.scalar} yields $T\in\C I_{\H}$. Therefore
$\beta_UT=UTU^*=UU^*T=T$. 
Thus, $T\neq 0$ implies $\beta_U=1$, and the proof is complete.
\end{proof}

\begin{remark}\label{scalar}
One can see from the proof of Lemma \ref{A_U.spectral} that $\Phi(\C I_{\H})=\C I_{\H}$.
\end{remark}

Denote by $\NHO$ the set of all nilpotent rank-one operators on $\H$, and by $\tzh$ the linear space of all finite-rank operators on $\H$ with vanishing trace.

\begin{lemma}\label{rank-one.preserving}
    $\Phi(R)\in\NHO$ for every $R\in\NHO$.
\end{lemma}

\begin{proof}
Since Lemma \ref{A_U.spectral} shows, in particular, that $\Phi$ transforms unitary similarity into similarity, a minor modification of the proof of \cite[Lemma 3.4 (i)]{AmaraOudghiri} shows that, for every $R\in\NHO$, the operator $\Phi(R)$ is rank-one. Moreover, since for every unimodular $\alpha\in\C$, $R\sim_u\alpha R$, we get $\Phi(R)\sim\alpha \Phi(R)$. Therefore,  $\sigma\left(\Phi(R)\right)=\alpha \sigma\left(\Phi(R)\right)$ for every unimodular $\alpha\in\C$. Since $\sigma\left(\Phi(R)\right)$ is finite, we get $\sigma\left(\Phi(R)\right)=\{0\}$, and consequently $\Phi(R)$ is also nilpotent.
\end{proof}

\begin{remark}\label{phi.tzh}
Since $\tzh$ is spanned by rank-one nilpotent operators, we can see from Lemma \ref{rank-one.preserving} that $\Phi(\tzh)\subseteq\tzh$.
\end{remark}

\subsection{The finite-dimensional case}

We can now prove Theorem \ref{main.thm.1}.

\begin{proof}[Proof of Theorem \ref{main.thm.1}]
{\rm (ii)}$\implies${\rm (iii)}. By Remark \ref{phi.tzh}, 
$\Phi(\tzh)\subseteq\tzh$.
Moreover, since $\Phi$ is injective and $\tzh$ is finite-dimensional, it follows that $\Phi(\tzh)=\tzh$. Hence the restriction $\Phi:\tzh\to\tzh$ is bijective. By \cite[Lemma 2.4]{LiPierce}, there exist an invertible operator $A\in\BH$ and a nonzero $c\in\C$ such that either
\begin{equation}\label{cAFA-1}
\Phi(F)=cAFA^{-1}\qquad\text{for every }F\in\tzh,
\end{equation}
or
\begin{equation}\label{cAFtA-1}
\Phi(F)=cAF^{\tr}A^{-1}\qquad\text{for every }F\in\tzh. 
\end{equation}
By Remark \ref{scalar}, there exists $\lambda\in\C$ such that $ \Phi(I_{\H})=\lambda I_{\H}$. Let $X\in\BH$ and set $F=X-\frac{{\rm Tr}(X)}{n}I_{\H}$ where $n=\dim\H$. Then ${\rm Tr}(F)=0$, so $F\in\tzh$, and
$$
X=F+\frac{{\rm Tr}(X)}{n}I_{\H}.
$$
Suppose first that \eqref{cAFA-1} holds. Then
\begin{align*}
\Phi(X) &=cA\left(X-\frac{{\rm Tr}(X)}{n}I_{\H}\right)A^{-1} +\frac{{\rm Tr}(X)}{n}\lambda I_{\H}\\
&=cAXA^{-1} +\frac{\lambda-c}{n}{\rm Tr}(X)I_{\H}. 
\end{align*}
Similarly, if \eqref{cAFtA-1} holds, then
\begin{align*}
\Phi(X) &=cA\left(X-\frac{{\rm Tr}(X)}{n}I_{\H}\right)^{\tr}A^{-1} +\frac{{\rm Tr}(X)}{n}\lambda I_{\H}\\
&=cAX^{\tr}A^{-1} +\frac{\lambda-c}{n}{\rm Tr}(X)I_{\H}.
\end{align*}
This gives the desired conclusion.
\end{proof}

\subsection{The infinite-dimensional case}

In the remainder of this section, $\H$ is assumed to be infinite-dimensional.

\medskip

For vectors $x,y\in\H$, let $x\otimes y$ denote the bounded linear operator defined on $\H$ by $(x\otimes y)h=\langle h,y\rangle x$.

\begin{remark}\label{structure.preservers.NHO}
In \cite[Proposition 3.1]{Semrl}, \v{S}emrl proved that if
$\varphi:\BH\to\BH$ is an injective linear map satisfying
$\varphi(\NHO)\subseteq\NHO$, then one of the following alternatives holds:
\begin{enumerate}[\rm (i)]
    \item There exist a nonzero vector $x\in\H$ and an injective conjugate-linear map
    $h:\tzh\to\{x\}^{\perp}$ such that
    $$
    \varphi(F)=x\otimes h(F)\qquad\text{for every }F\in\tzh.
    $$
    
    \item There exist a nonzero vector $x\in\H$ and an injective linear map
    $h:\tzh\to\{x\}^{\perp}$ such that
    $$
    \varphi(F)=h(F)\otimes x\qquad\text{for every }F\in\tzh.
    $$
    
    \item There exist injective linear maps $A,B:\H\to\H$ such that
    $$
    \varphi(x\otimes y)=Ax\otimes By
    \qquad\text{for every }x\otimes y\in\NHO.
    $$
    
    \item There exist injective conjugate-linear maps $A,B:\H\to\H$ such that
    $$
    \varphi(x\otimes y)=Ay\otimes Bx
    \qquad\text{for every }x\otimes y\in\NHO.
    $$
\end{enumerate}
Although \v{S}emrl stated the result for separable infinite-dimensional Hilbert spaces, it is easy to see from his proof that separability is not used. This was also observed by Lu and Peng, who noted that the argument remains valid for arbitrary infinite-dimensional Banach spaces; see \cite[Proposition 2.6]{LuPeng}. Hence, no separability assumption is needed here.
\end{remark}

\begin{lemma}\label{A.and.B}
There exist either injective linear maps $A:\H\to\H$ and $B:\H\to\H$ such that
\begin{equation}\label{AxoBy}
\Phi(x\otimes y)=Ax\otimes By\quad \text{for every $x\otimes y\in\NHO$}
\end{equation}
or injective conjugate-linear maps $A:\H\to\H$ and $B:\H\to\H$ such that
\begin{equation}\label{AyoBx}
\Phi(x\otimes y)=Ay\otimes Bx\quad \text{for every $x\otimes y\in\NHO$.}
\end{equation}
\end{lemma}

\begin{proof}
Assume for a contradiction that $\Phi$ satisfies neither \eqref{AxoBy} nor \eqref{AyoBx}. Then, by Remark \ref{structure.preservers.NHO}, there exist a nonzero vector $x\in\H$ and a map $h:\tzh\to \{x\}^{\perp}$ such that either
$$
\Phi(F)=x\otimes h(F)\qquad\text{for every }F\in\tzh
$$
or
\begin{equation}\label{h(F)ox}
\Phi(F)=h(F)\otimes x\qquad\text{for every }F\in\tzh.    
\end{equation}
Assume that the former case holds, let $z\in\H$ be linearly independent of $x$, and let $T\in\BH$ be such that $\Phi(T)=z\otimes x$.
Let $P$ be a finite-rank orthogonal projection and put $U=I_{\H}-2P$.
Then $U$ is unitary and $UTU^*\simu T$. Hence $\Phi(UTU^*)\sim \Phi(T)=z\otimes x$. In particular, $\rank \Phi(UTU^*)=1$.

Set $N=UTU^*-T$. Since $N=-2\big(PT(I_{\H}-P)+(I_{\H}-P)TP\big)$, we have $N\in\tzh$. Therefore, $\Phi(N)=x\otimes y$ where $y=h(N)\in \{x\}^\perp$. Thus
$$
\Phi(UTU^*)
=
z\otimes x+x\otimes y.
$$
Since $z$ and $x$ are linearly independent and $\rank \Phi(UTU^*)=1$, the vectors
$x$ and $y$ must be linearly dependent. But $y\in\{ x \}^{\perp}$; so $y=0$. Thus $\Phi(N)=0$, and therefore $N=0$. Hence
$
(I_{\H}-2P)T(I_{\H}-2P)^*=T;
$
equivalently $PT=TP$. Since $P$ was arbitrary, we conclude that $T\in\C I_{\H}$. But by Remark \ref{scalar}, we have $\Phi(\C I_{\H})=\C I_{\H}$, contradicting $\rank \Phi(T)=1$.

The case \eqref{h(F)ox} is treated similarly by considering the rank-one operator $x\otimes z$ instead of $z\otimes x$.
\end{proof}

\begin{lemma}\label{rank.preserving}
The map $\Phi$ preserves the rank on $\tzh$.
\end{lemma}

\begin{proof}
Let $F\in\tzh$, and let us show that $\rank \Phi(F)=\rank F$. The equality is trivial if $F=0$, so assume that $F\neq0$. Since $F$ has finite rank, there exists a finite-dimensional
subspace $M\subseteq\H$ reducing $F$ such that
$$
F=\begin{bmatrix}
F_0 & 0\\
0 & 0
\end{bmatrix}
\begin{array}{l}
     M \\
     M^\perp 
\end{array}
$$
for some $F_0\in\mathcal{B}(M)$. Moreover, ${\rm Tr}(F_0)={\rm Tr}(F)=0.$
Then, according to \cite{Fillmore}, there exist an invertible operator
$S\in\mathcal B(M)$ and an orthonormal basis $\{e_1,\ldots,e_p\}$ of
$M$ such that $S^{-1}F_0S$ has zero diagonal entries with respect to this basis. For $i\in\{1,\ldots, p\}$, set
$v_i=S^{-1}F_0Se_i$. Then
$$
S^{-1}F_0S=\sum_{i=1}^p v_i\otimes e_i
$$
with
$
\langle v_i,e_i\rangle = \left\langle S^{-1}F_0S e_i,e_i\right\rangle=0.
$
Consequently,
\begin{equation}\label{nilpotent.representation}
F=\sum_{i=1}^p u_i\otimes f_i,
\end{equation}
where $u_i=Sv_i$ and $f_i=(S^{-1})^*e_i$. Note that, for every $i\in\{1,\ldots, p\}$,
\begin{equation}\label{nilpotent.representation2}
  \langle u_i,f_i\rangle
=
\langle Sv_i,(S^{-1})^*e_i\rangle
=
\langle v_i,e_i\rangle
=
0.
\end{equation}
Moreover, since $(S^{-1})^*$ is invertible, the vectors $f_1,\ldots,f_p$ are linearly independent. Therefore, elementary linear algebra shows that 
$\ran F= \spn\{u_1,\ldots, u_p\}$, and so
$$
\rank F=\dim\spn\{u_1,\ldots,u_p\}.
$$

Assume first that \eqref{AxoBy} holds. Then combining with \eqref{nilpotent.representation} and \eqref{nilpotent.representation2} gives $\Phi(F)=\sum_{i=1}^p Au_i\otimes Bf_i$. Since $B$ is injective, the vectors $Bf_1,\ldots,Bf_p$ are linearly independent, and therefore $\ran \Phi(F)=\spn\{Au_1,\ldots,Au_p \}$, implying
$$
\rank\Phi(F)=\dim\spn\{Au_1,\ldots,Au_p\}.
$$
Since $A$ is injective,
$$
\dim\spn\{Au_1,\ldots,Au_p\}
=
\dim\spn\{u_1,\ldots,u_p\}
=
\rank F.
$$
Thus $\rank\Phi(F)=\rank F$.

Now, if \eqref{AyoBx} holds, then
$$
\Phi(F)=\sum_{i=1}^p Af_i\otimes Bu_i,
$$
and, in this case, the vectors $Af_1,\ldots,Af_p$ are linearly independent, and the same argument applied to $\Phi(F)^*$
gives
\begin{align*}
\rank\Phi(F)=\rank\Phi(F)^*
&=
\dim\spn\{Bu_1,\ldots,Bu_p\}\\
&=
\dim\spn\{u_1,\ldots,u_p\}
=
\rank F.    
\end{align*}
Hence $\Phi$ preserves the rank on $\tzh$.
\end{proof}

\begin{lemma}\label{standard.form.F0}
There exist an invertible operator $A\in\BH$ and a nonzero $c\in\C$ such that either
\begin{equation}\label{form1}
   \Phi(F)=c AFA^{-1}\quad\text{for every }F\in\tzh 
\end{equation}
or
\begin{equation}\label{form2}
\Phi(F)=c AF^{\tr}A^{-1}\quad\text{for every }F\in\tzh.    
\end{equation}
\end{lemma}

\begin{proof}
    First, we show that $\Phi(\tzh)=\tzh$. We already know that
    $$
    \Phi(\tzh)\subseteq\tzh.
    $$
    Let $R\in\BH$ be a rank-one operator and set
$T=\Phi^{-1}(R)$. Let $P\in\FH$ be an orthogonal projection and put
$U=I_{\H}-2P$. Then, $UTU^*\simu T$, and so $\Phi(UTU^*)\sim \Phi(T)=R$. Hence, $\Phi(UTU^*)$ has rank one. Set $N=UTU^*-T$. Then,
$$
N=-2
\begin{bmatrix}
0&PT_{|\ran(I_{\H}-P)}\\
(I_{\H}-P)T_{|\ran P}&0
\end{bmatrix}
\begin{array}{l}
     \ran P \\
     \ran (I_{\H}-P) 
\end{array},
$$
and so $N\in\tzh$ and
$$
\rank N =\rank PT(I_{\H}-P) + \rank (I_{\H}-P)TP.
$$
Note that $\Phi(N)=\Phi(UTU^*)-\Phi(T)$ has rank at most $2$. By Lemma \ref{rank.preserving}, $\rank N\le2$. Since $P$ was arbitrary, Proposition \ref{rank.proposition} gives $\Phi^{-1}(R)=T\in \FH+\C I_{\H}$. Since rank-one operators span $\FH$ and $\Phi(\C I_{\H})=\C I_{\H}$, we get 
$$
\Phi^{-1}(\FH+\C I_{\H})\subseteq \FH+\C I_{\H}.
$$
We now use the codimension argument employed by \v{S}emrl
\cite[p.~80]{Semrl}. Consequently,
\begin{equation}\label{vector.space.sequence}
\Phi(\tzh)\subseteq \tzh \subset \FH \subset \FH+\C I_{\H} \subseteq \Phi(\FH+\C I_{\H}).
\end{equation}
Since $\tzh$ has codimension one in $\FH$ and $\FH$ has codimension one in $\FH+\C I_{\H}$, bijectivity of $\Phi$ shows that $\Phi(\tzh)$ has codimension two in $\Phi(\FH+\C I_{\H})$. Therefore, the inclusions in \eqref{vector.space.sequence} force $\Phi(\tzh)=\tzh$.

\medskip

Now, let $A$ and $B$ be the injective maps obtained in Lemma \ref{A.and.B}. Since $\Phi(\tzh)=\tzh$, $\Phi$ preserves the rank on $\tzh$, and taking into account that the set of rank-one operators in $\tzh$ is exactly $\NHO$, we have $\Phi\left(\NHO \right)=\NHO$. Hence, one can easily see from \eqref{AxoBy} and \eqref{AyoBx} that $A$ and $B$ are also surjective. Consequently, by \cite[Remark 3.2]{Semrl} (see also \cite[Remark 2.7]{LuPeng}), the map $\Phi:\tzh\to\tzh$ has one of the forms in \eqref{form1} and \eqref{form2}.
\end{proof}

\begin{proof}[Proof of Theorem \ref{main.thm.2}]
(ii)$\implies$(iii). We treat only the case \eqref{form1} and we will show that $\Phi$ has the form in \eqref{standard.form1.spectral}. In the other case \eqref{form2}, we replace $\Phi$ with $X\mapsto \Phi(X)^{\tr}$ which still satisfies condition (ii) of Theorem \ref{main.thm.2} and satisfies \eqref{form1}, implying that $\Phi$ has the form in \eqref{standard.form2.spectral}.

Define a map $\Psi:\BH\to\BH$ by $\Psi(X)=c^{-1}A^{-1}\Phi(X)A$. We will show that $\Psi=I_{\BH}$ which implies that $\Phi$ has the form in \eqref{standard.form1.spectral}. Note that
\begin{equation*}
    \Psi(F)=F \quad\text{for every } F\in\tzh.
\end{equation*}
Moreover, if $U\in\UH$, then Lemma \ref{A_U.spectral} gives, for every $X\in\BH$, 
\begin{equation}\label{S_U}
\Psi(UXU^*)=c^{-1}A^{-1}\Phi(UXU^*)A=c^{-1}A^{-1}A_U\Phi(X)A_U^{-1}A=S_U\Psi(X)S_U^{-1}
\end{equation}
where $S_U=A^{-1}A_UA$. Then, for every $F\in\tzh$, we have
$$
UFU^*=\Psi(UFU^*)=S_U F S_U^{-1};
$$
that is $S_U^{-1}UF=FS_U^{-1}U$. Since the commutant of $\tzh$ is $\C I_{\H}$, we get $S_U^{-1}U=\alpha_U I$ for some nonzero $\alpha_U\in\C$; that is $S_U=\alpha_U^{-1} U$. Substituting in \eqref{S_U}, we obtain
\begin{equation}\label{Psi.unitary.similarity}
\Psi(UXU^*)=U\Psi(X)U^*\quad\text{for every }X\in\BH.    
\end{equation}

Next, we show that $\Psi=I_{\BH}$. For convenience, define a map $\Gamma:\BH\to\BH$ by $\Gamma(X)=\Psi(X)-X$. Clearly, $\Gamma$ vanishes on $\tzh$ and, for all $U\in\UH$ and $X\in\BH$, 
\begin{equation}\label{Gamma.unitary.similarity.invariant}
\Gamma(UXU^*)=\Psi(UXU^*)-UXU^*=U\Psi(X)U^*-UXU^*=U\Gamma(X)U^*.
\end{equation}
Let $P\in\BH$ be a finite-rank orthogonal projection, and set $V=I_{\H}-2P$. Then $V$ is unitary and, for every $X\in\BH$, the operator $VXV^*-X$ belongs to $\tzh$. It follows that
$$
\Gamma(VXV^*)=\Gamma(X)\quad\text{for every $X\in\BH$},
$$
and hence combining with \eqref{Gamma.unitary.similarity.invariant} gives $V\Gamma(X)=\Gamma(X)V$; equivalently, $P\Gamma(X)=\Gamma(X)P$. Since $P$ was arbitrary, $\Gamma(X)\in\C I_{\H}$, and therefore there exists a linear functional $f:\BH\to\C$ such that $\Gamma(X)=f(X) I_{\H}$ for every $X\in\BH$. Moreover, it follows from \eqref{Gamma.unitary.similarity.invariant} that
$$
f(UXU^*)=f(X) \qquad\text{for all $U\in\UH$ and $X\in\BH$};
$$
equivalently,
$
f(UX)=f(XU).
$
Since $\BH$ is spanned by $\UH$, we obtain
$$
f(TX-XT)=0 \quad\text{for all $T,X\in\BH$}.
$$
Since every operator in $\BH$ is a finite sum of commutators (i.e., operators of the form $TS-ST$) \cite{Halmos} (see also \cite[Remark 3.3]{LuPeng}), the functional $f$ is identically zero. Consequently, $\Psi=I_{\BH}$, as desired.
\end{proof}

\section{Proof of Theorems \ref{main.thm.3} and \ref{main.thm.4}}

Throughout this section, $\Phi:\BH\to\BH$ is a bijective linear map satisfying
$$
T\simu S\implies \w(\Phi(T))=\w(\Phi(S))
\qquad\text{for all }T,S\in\BH.
$$

\begin{lemma}\label{A_U.numerical}
For every $U\in\UH$, there exists $W_U\in\UH$ such that
$$
\Phi(UXU^*)=W_U\Phi(X)W_U^*
\qquad\text{for every }X\in\BH.
$$
\end{lemma}

\begin{proof}
Fix $U\in\UH$, and choose $V\in\UH$ such that $V^2=U$. Define
$\varphi_V:\BH\to\BH$ by
$$
\varphi_V(X)=\Phi\bigl(V\Phi^{-1}(X)V^*\bigr).
$$
Then $\varphi_V$ is a bijective linear map satisfying
$$
w(\varphi_V(X))=w(X)
\qquad\text{for every }X\in\BH.
$$
By \cite{Chan}, there exist a unitary operator $S_V\in\UH$ and a
unimodular scalar $\alpha_V\in\C$ such that either
$$
\varphi_V(X)=\alpha_VS_VXS_V^*
\qquad\text{for every }X\in\BH,
$$
or
$$
\varphi_V(X)=\alpha_VS_VX^{\tr}S_V^*
\qquad\text{for every }X\in\BH.
$$
Applying the corresponding identity twice, as in the proof of
Lemma \ref{A_U.spectral}, we obtain a unitary operator $W_U\in\UH$
and a unimodular scalar $\beta_U\in\C$ such that
$$
\Phi(UXU^*)=\beta_UW_U\Phi(X)W_U^*
\qquad\text{for every }X\in\BH.
$$
Finally, the argument used at the end of the proof of
Lemma \ref{A_U.spectral}, applied to $T=\Phi^{-1}(I_{\H})$, yields
$\beta_U=1$. Hence
$$
\Phi(UXU^*)=W_U\Phi(X)W_U^*
\qquad\text{for every }X\in\BH.
$$
\end{proof}

Next, we prove Theorem \ref{main.thm.3}.

\begin{proof}[Proof of Theorem \ref{main.thm.3}]
{\rm (ii)}$\implies${\rm (iii)}. By Lemma
\ref{A_U.numerical}, $\Phi$ preserves unitary similarity. The desired conclusion now follows from \cite[Theorem 3.6]{HornLiTsing}, together with the bijectivity of $\Phi$.
\end{proof}

We conclude with the proof of Theorem \ref{main.thm.4}.

\begin{proof}[Proof of Theorem \ref{main.thm.4}]
{\rm (ii)}$\implies${\rm (iii)}. Since $\Phi$ preserves unitary similarity, it satisfies
$$
T\simu S\implies \Phi(T)\sim\Phi(S)
\qquad\text{for all }T,S\in\BH.
$$
Hence, by Theorem \ref{main.thm.2}, there exist an invertible operator
$A\in\BH$ and a nonzero scalar $c\in\C$ such that either
\begin{equation}\label{cAXA-1}
\Phi(X)=cAXA^{-1}
\qquad\text{for every }X\in\BH,  
\end{equation}
or
\begin{equation}\label{cAXtA-1}
\Phi(X)=cAX^{\tr}A^{-1}
\qquad\text{for every }X\in\BH.  
\end{equation}
It remains to show that $A$ is a nonzero scalar multiple of a unitary
operator.

Fix $U\in\UH$. By Lemma \ref{A_U.numerical}, there exists
$W_U\in\UH$ such that
$$
\Phi(UXU^*)=W_U\Phi(X)W_U^*
\qquad\text{for every }X\in\BH.
$$

Assume first that \eqref{cAXA-1} holds. Then, for every $X\in\BH$,
$$
AUXU^*A^{-1}
=c^{-1}\Phi(UXU^*)=c^{-1}W_U\Phi(X)W_U^*=
W_UAXA^{-1}W_U^*,
$$
or equivalently,
$$
(AU)X(AU)^{-1}
=
(W_UA)X(W_UA)^{-1}
\qquad\text{for every }X\in\BH.
$$
It follows that $(W_UA)^{-1}AU$ commutes with every operator in
$\BH$. Therefore,
$$
AU=\alpha_UW_UA
$$
for some nonzero scalar $\alpha_U\in\C$. Since $U$ and $W_U$ are
unitary,
$$
\|A\|=\|AU\|
=|\alpha_U|\,\|W_UA\|
=|\alpha_U|\,\|A\|,
$$
and hence $|\alpha_U|=1$. Consequently,
$$
U^*A^*AU
=
(\overline{\alpha_U}A^*W_U^*)
(\alpha_UW_UA)
=
A^*A
$$
for every $U\in\UH$. Thus $A^*A$ commutes with every unitary operator
in $\BH$, and hence there exists $r>0$ such that $A^*A=r^2I_{\H}$; equivalently,
$\left(r^{-1}A\right)^*\left(r^{-1}A\right)=I_{\H}$.
Since $r^{-1}A$ is invertible, we also have $\left(r^{-1}A\right)\left(r^{-1}A\right)^*=I_{\H}$. Therefore, $A=rW$ for some $W\in\UH$.

Now assume that \eqref{cAXtA-1} holds. Then, for every $X\in\BH$,
$$
A(U^*)^{\tr}X^{\tr}U^{\tr}A^{-1}
=c^{-1}\Phi(UXU^*)=c^{-1}W_U\Phi(X)W_U^*=
W_UAX^{\tr}A^{-1}W_U^*.
$$
As $X^{\tr}$ ranges over $\BH$, it follows, as above, that
$$
A(U^*)^{\tr}=\alpha_UW_UA
$$
for some nonzero scalar $\alpha_U\in\C$. Again, taking norms gives
$|\alpha_U|=1$. Hence
$$
\bigl((U^*)^{\tr}\bigr)^*A^*A(U^*)^{\tr}=A^*A
\qquad\text{for every }U\in\UH.
$$
Since $U\mapsto (U^*)^{\tr}$ is a bijection of $\UH$ onto itself, $A^*A$ commutes with every
unitary operator in $\BH$. Thus $A^*A=r^2I_{\H}$
for some $r>0$, and, as before, $A=rW$ for some $W\in\UH$.
\end{proof}


\begin{thebibliography}{99}

\bibitem{AmaraOudghiri}
Z. Amara and M. Oudghiri,
{\it Spectral function preservation by bijective linear maps on unitarily similar operators},
J. Math. Anal. Appl. {\bf 557} (2026), 130330.

\bibitem{BourhimMabrouk}
A. Bourhim and M. Mabrouk,
{\it Unitary similarity and the numerical radius preservers},
Linear Algebra Appl. {\bf 714} (2025), 15--27.

\bibitem{BresarSemrl}
M. Bre\v{s}ar and P. \v{S}emrl,
{\it Linear maps preserving the spectral radius},
J. Funct. Anal. {\bf 142} (1996), no. 2, 360--368.

\bibitem{Chan}
J.-T. Chan,
{\it Numerical radius preserving operators on $B(H)$},
Proc. Amer. Math. Soc. {\bf 123} (1995), 1437--1439.

\bibitem{Costara}
C. Costara,
{\it Linear bijective maps preserving spectral functions on pairs of similar operators},
J. Math. Anal. Appl. {\bf 530} (2024), no. 2, 127683.

\bibitem{DavidsonMarcoux}
K. R. Davidson and L. W. Marcoux,
{\it Linear spans of unitary and similarity orbits of a Hilbert space operator},
J. Operator Theory {\bf 52} (2004), 113--132.

\bibitem{Fillmore}
P. A. Fillmore,
{\it On similarity and the diagonal of a matrix},
Amer. Math. Monthly {\bf 76} (1969), no. 2, 167--169.

\bibitem{Halmos}
P. R. Halmos,
{\it Commutators of operators},
Amer. J. Math. {\bf 74} (1952), no. 1,  237--240.

\bibitem{Hiai}
F. Hiai,
{\it Similarity preserving linear maps on matrices},
Linear Algebra Appl. {\bf 97} (1987), 127--139.

\bibitem{HornLiTsing}
R. A. Horn, C.-K. Li and N.-K. Tsing,
{\it Linear operators preserving certain equivalence relations on matrices},
SIAM J. Matrix Anal. Appl. {\bf 12} (1991), no. 2, 195--204.

\bibitem{KarderPetekTaghavi}
M. Karder, T. Petek and A. Taghavi,
{\it Unitary similarity preserving linear maps on $B(H)$},
Integral Equations Operator Theory {\bf 82} (2015), 51--60.

\bibitem{LiPierce}
C.-K. Li and S. Pierce,
{\it Linear operators preserving similarity classes and related results},
Canad. Math. Bull. {\bf 37} (1994), 374--383.


\bibitem{Lim}
M. H. Lim,
{\it A note on similarity preserving linear maps on matrices},
Linear Algebra Appl. {\bf 190} (1993), 229--233.

\bibitem{LuPeng}
F. Lu and C. Peng,
{\it Similarity-preserving linear maps on $B(X)$},
Studia Math. {\bf 209} (2012), no. 1, 1--10.

\bibitem{Petek}
T. Petek,
{\it A note on unitary similarity preserving linear mappings on $B(H)$},
Linear Algebra Appl. {\bf 394} (2005), 217--224.

\bibitem{Semrl}
P. \v{S}emrl,
{\it Similarity preserving linear maps},
J. Operator Theory {\bf 60} (2008), 71--83.

\end{thebibliography}
\end{document}